\documentclass[11pt]{amsart}
\usepackage{amsfonts,amsmath,mathrsfs,amssymb,amsthm,amscd,latexsym}

\usepackage[usenames]{color}
\usepackage[all]{xy}
\usepackage{enumerate}
\xyoption{all}
\usepackage{srcltx} 

\newtheorem{thm}{Theorem}[section]
\newtheorem{lem}[thm]{Lemma}
\newtheorem{cor}[thm]{Corollary}
\newtheorem{prop}[thm]{Proposition}
\newtheorem{thmABC}{Theorem}[section]

\theoremstyle{definition}
\newtheorem{rem}[thm]{Remark}
\newtheorem{defn}[thm]{Definition}

\newcommand{\dbc}[1]{\mathbf{D}^b(#1)}
\newcommand{\what}[1]{{\widehat #1}}
\newcommand{\cE}{{\mathcal E}}
\newcommand{\cF}{{\mathcal F}}
\newcommand{\cG}{{\mathcal G}}
\newcommand{\cI}{{\mathcal I}}
\newcommand{\cL}{{\mathcal L}}
\newcommand{\cO}{{\mathcal O}}
\newcommand{\tC}{{\widetilde C}}
\newcommand{\mL}{{L}^{-}}
\newcommand{\pL}{L^{+}}
\newcommand{\mX}{X^{-}}
\newcommand{\pX}{X^{+}}
\newcommand{\mP}{P^{-}}
\newcommand{\pP}{P^{+}}
\newcommand{\Ro}{\mathcal R^{0}}
\newcommand{\Ru}{\mathcal R^{1}}
\newcommand{\bR}{{\mathbf R}}
\DeclareMathOperator{\rk}{rk}

\DeclareMathOperator{\Pic}{Pic}
\DeclareMathOperator{\codim}{codim}
\DeclareMathOperator{\supp}{supp}
\DeclareMathOperator{\Sing}{Sing}

\newcommand{\fm}{\mathscr S}

\title{Theta-duality on Prym varieties and a Torelli Theorem}

\author{Mart\'i Lahoz}
\address{Universitat de Barcelona \\
Departament d'\`Algebra i Geometria \\
Facultat de Matem\`atiques \\
Gran Via 585 \\
08007 Barcelona, Spain }
\email{marti.lahoz@ub.edu}
\curraddr{Mathematisches Institut, Universit\"at Bonn, Endenicher Allee 60, 53115 Bonn, Germany.}

\author{Juan Carlos Naranjo}
\address{Universitat de Barcelona \\
Departament d'\`Algebra i Geometria \\
Facultat de Matem\`atiques \\
Gran Via 585 \\
08007 Barcelona, Spain }
\email{jcnaranjo@ub.edu}

\thanks{Both authors have been partially supported by the Proyecto de Investigaci\'on MTM2009-14163-C02-01. This paper was revised while the first named author was supported by the SFB/TR 45 `Periods, Moduli Spaces and Arithmetic of Algebraic Varieties’
of the DFG (German Research Foundation)}

\begin{document}

\begin{abstract}
Let $\pi: \tC \to C$ be an unramified double covering of irreducible smooth curves and let $P$ be the attached 
Prym variety. We prove the scheme-theoretic theta-dual equalities in the Prym variety $T(\tC)=V^2$ and $T(V^2)=\tC$,
where $V^2$ is the Brill-Noether locus of $P$ associated to $\pi$ considered by Welters. As an
application we prove a Torelli Theorem analogous to the fact that the symmetric product $D^{(g)}$ of a curve $D$
of genus $g$ determines the curve.
\end{abstract}

\maketitle

\pagestyle{plain}

\section*{Introduction}

Prym varieties are principally polarized abelian varieties (ppav in the sequel) naturally attached to \'etale double
coverings of curves. In analogy with the Jacobian case there exists two varieties of divisors $\pX,\mX$ playing the role
of the symmetric products $C^{(g-1)}$, $C^{(g)}$ for a Jacobian variety $JC$ of a curve of genus $g$. Indeed both have
natural maps to (torsors of) the Prym variety and these maps are the corresponding Albanese maps. The analogy
mentioned above comes from the behaviours of these morphisms: while the image of $\pX$ is a translate of the Prym
Theta divisor $\Xi$, $\mX$ maps birationally onto the Prym variety. In \cite{Nar} and \cite{SV}
it is proved that (under some mild conditions) the variety $\pX$ determines the covering. One of the goals of this
paper is to prove that this is so for $\mX$.
 
This connects with the geometry of the Brill-Noether loci $V^i$ of the Prym variety (introduced by Welters \cite[(1.2)]{WENS})
and its scheme-theoretic theta-dual. To give an idea $V^1$ is a canonical model of Theta divisor $\Xi$, while $V^3$ correspond
to the stable singularities of $\Xi$. Instead, $V^2$ has not such a natural description in terms of the Theta divisor.
It is remarkable that $V^2$ is exactly the subscheme of $P$ where the map $\mX \to P$ has positive dimensional fibers. 
Following Pareschi and Popa (see \cite[Def. 4.2]{PPminimal}), we use the Fourier-Mukai transform to provide a natural scheme structure 
on the theta-dual
set $T(X)=\{a\in P \,\vert \, a+X\subset \Xi\}.$ Actually the definition of $T(Y)$ may not deserve the name (already
established) of ``theta-dual'' since the natural inclusion $Y\subset T(T(Y))$ is usually not an equality. 
The computation of theta-dual sets is a classical subject for
Jacobians and Pryms. For example, the equality $T(W_d)=W_{g-d-1}$ is well-known and Welters
(\cite{Wacta,Wrecovering}) has studied $T(\Sing(\Theta ))$ for Jacobians and for generic Prym varieties (of high enough
dimension in the second case). In particular, for a generic Prym variety, he proved that the component of dimension $2$
of $T(V^3)=T(\Sing(\Xi))$ is the surface $\tC-\tC$. Debarre proved for the covering of a generic tetragonal curve of high genus
one has the equality  $T(V^3)=T(\Sing_{st}(\Xi))=\tC-\tC$  (see \cite[Cor. 4.6]{Debarre}).

What we prove in this paper are the following theorems:

\begin{thmABC}\label{ThmA}
Let $\tC\to C$ be an irreducible unramified double covering of a smooth complete non-hyperelliptic
irreducible curve of genus $g\geq 4$, and let $(P,\Xi)$ be its Prym variety. Then, the scheme-theoretic equalities
$T(\tC)=V^2$ and $T(V^2)=\tC$ hold.
\end{thmABC}

For Prym varieties of dimension $4$ ($g=5$), a set-theoretic version of the previous theorem is proven in \cite[Thm. 1]
{Izadi}.

\begin{thmABC} \label{ThmB}
Let $\tC\to C$ be an irreducible unramified double covering of a smooth complete non-hyperelliptic
irreducible curve of genus $g>4$, and let $(P,\Xi)$ be its Prym variety. 
Then the variety $\mX (\tC,C)$ determines the covering
$\tC\to C$.
\end{thmABC}

In fact, we will see that the principal polarization $\Xi$ and $V^2$ determine the covering and $\mX$ determines $V^2$ and $\Xi$. 
We point out that the principal polarization does not determine in general the covering (see \cite[Prop. 1.1]{donagi} and 
\cite[Theorem 0.1]{IL}).

Recall that when $C$ is hyperelliptic the Prym variety is a Jacobian or a product of Jacobians (see \cite[pg. 344]{Mumford_Pryms}) 
and the statements above are uninteresting from the point of view of Prym theory.

A few words about the method of the proofs: fixed an unramified double covering
$\tC \to C$ ($C$ non-hyperelliptic of genus $g$), there is an embedding in the Prym variety $j: \tC
\to P$. Analogously with the Jacobian case one can consider the Fourier-Mukai transform of $j_{\ast}L$,
where $L \in \Pic^{2g-2}(\tC)$ has norm $\omega_C$. When $h^0(\tC,L)$ is even this transform is (quasi-isomorphic to) a
sheaf whose projectivization is $\pX$ (see \cite[Thm. 2.1]{Nar}). 
Instead, when $h^0(L)$ is odd the Fourier-Mukai transform $\bR \fm_P(j_{\ast}L)$ is a complex
with two non-vanishing cohomology sheaves $\Ro$ and $\Ru$. One of the main points of the paper is the
isomorphism $\Ru \cong \cI_{V^2}(\Xi_{\lambda })$. 
The scheme-theoretic equality $T(V^2)=\tC$ follows quickly (and also the Torelli-like Theorem corresponding
to $\mX$). To prove the isomorphism $\Ru \cong \cI_{V^2}(\Xi_{\lambda })$ we combine the use of standard
properties of the Fourier-Mukai transform with the opposite equality $T(\tC)=V^2$. Finally, the proof of this last
statement (very easy set-theoretically) follows from a delicate argument involving Fitting supports.

It is quite surprising, at least for the authors, that Brill-Noether locus $V^2$ determines 
the covering without genericity hypothesis or restrictions on the genus. 
In the last section, this fact allows us to illustrate Theorem \ref{ThmA} in a very well-known context: the intermediate
Jacobian $J$ of a smooth cubic threefold $V$, since $J$ is isomorphic (as ppav) to a Prym variety of dimension $5$.

\vskip 3mm
{\bf Acknowledgments.}
We wish to thank Andreas H\"oring for sharing with us his work \cite{Ho2}. We would also thank Miguel \'Angel Barja
for many valuable discussions, especially while carrying out the first author Ph.D. thesis. The second author is
grateful to Alberto Collino and Gian Pietro Pirola for stimulating conversations on the geometry of
the cubic threefolds. Finally, we thank the referee for many suggestions that improved the exposition of the paper.

\section{Notation and preliminaries}

All the varieties are defined over an algebraically closed field of characteristic $\ne$ 2 . 

\subsection{Prym varieties}
We recall some basic facts of the theory of Prym varieties. We quote
\cite{Mumford_Pryms} for the details.

Let $\pi: \tC\to C$ be an irreducible unramified double covering
of an irreducible projective smooth curve $C$ of genus $g$. The kernel of the
norm map 
\begin{equation*}
Nm_\pi :J\tC\to JC
\end{equation*}
 has two irreducible
components. The component containing the origin is an abelian variety, 
$P=P(\tC,C)$ of dimension $g-1$, called the
{\slshape Prym variety\/} of the covering. We denote by $P'$ the other component. The
principal polarization on $J\tC$ restricts to twice a principal polarization on $P$, 
an effective divisor
representing the polarization (determined up to translation) is denoted by $\Xi .$
Therefore $(P,\Xi)$ is a ppav. Identifying $P$ with its dual using the principal polarization 
$\Xi$ or dualizying
$Nm_\pi$, we can consider also $P\subset \ker(\Pic^0\tC \to \Pic^0C)$.

We will fix a point $\tilde c \in \tC$ and we will denote by
$\sigma$ the involution in $\tC$ associated to the covering. We define
$j:\tC\to P$ to be the embedding given by $j(\tilde
x)=\cO_{\tC}(\tilde x-\sigma \tilde x+\tilde c-\sigma \tilde c)$.

The theta divisor can be defined canonically in 
\begin{equation*}
P^{+}=P^{+} (\tC,C):=\{L \in \Pic^{2g-2}(\tC) \, \vert \, Nm_{\pi}(L)=\omega_{C},\;
h^{0}(\tC,L)\text{ even} \},
\end{equation*}
as $\Xi^{+}=\{L\in P^{+}\,\vert \, h^{0}(\tC,L)>0\}.$
In all the paper an element $M^{+}\in P^{+}$ is fixed and $\Xi$ denotes the divisor in $P$ obtained
by translating $\Xi^{+}$ by $M^{+}$. In symbols: $\Xi=t^{\ast}_{M^+}(\Xi^{+})$, where $t_{M^+}:P^+\to P$ stands
for the isomorphism $t_{M^+}(L)=L\otimes (M^+)^{-1}$  in $\Pic\tC$.
 
The Jacobi Inversion Theorem for Prym varieties (see, e.g. \cite[Lemma 3.2]{Nar}, or \cite[III.2.4]{Tesi}) reads:
\begin{equation}\label{eq:JacobiInvPryms}
 j^{\ast}(\Xi)=\sigma^\ast(M^+)(\sigma \tilde c-\tilde c).
\end{equation}

Let $X^+=X^{+}(\tC,C)$, $X^-=X^{-}(\tC,C)$ be the varieties of special divisors defined by:
\begin{equation*}
\begin{aligned}
& X^+=\{D\in \tC^{(2g-2)} \,\vert \, \pi^{(2g-2)}(D)\in \vert \omega_C\vert \text{ and }
h^0(\tC,\cO_\tC(D)) 
\equiv 0\mod 2\}, \\
& X^-=\{D\in \tC^{(2g-2)} \,\vert \, \pi^{(2g-2)}(D)\in \vert \omega_C\vert \text{ and }
h^0(\tC,\cO_\tC(D)) 
\equiv 1\mod 2\}.
\end{aligned}
 \end{equation*}

Observe that the natural map $\tC^{(2g-2)}\to \Pic^{2g-2}(\tC)$ sends $X^+$ and $X^-$ to $P^+$ and
$P^-$ respectively. In the even case $X^+$ maps onto $\Xi^+$. 
The main result in \cite{Nar} and \cite{SV} is that, under some conditions, $X^+$ determines the
covering. Instead $X^-\to P^-$ is birational.

\begin{defn}\label{def:Vr} We define the scheme-theoretical Brill-Noether loci in Prym varieties following Welters (see \cite[(1.2)]{WENS}\footnote{We recall that De Concini and Pragacz \cite[Def. 1]{dCP}
define a more reduced natural scheme structure on 
$V^r$.}) 
\begin{align*} 
V^r&:=W_{g(\tC )-1}^r(\tC)\cap P^+ &\text{if }r\text{ is odd,}\\
V^r&:=W_{g(\tC )-1}^r(\tC)\cap P^- &\text{if }r\text{ is even.}
\end{align*}
\end{defn}

Observe that as a set,  if $r$ is even
(resp. odd), then $V^r\subset P^-$ (resp. $V^r\subset P^+$) is
\begin{equation*}V^r=\{L\in Nm_{\pi}^{-1}(\omega_{C})\,\vert \,h^0(L)\geq r+1,\: h^0(L)\equiv
r+1\mod 2\}.
\end{equation*}

For example, the first odd cases are $V^1=\Xi^+\subset P^+=V^{-1}$ and $V^3=\Sing_{st}(\Xi)\subset P^+$, 
the stable singularities of $\Xi$ (see \cite[pg. 343]{Mumford_Pryms}). The first even cases are 
$V^0=P^-$ and $V^2=T(\tC)\subset P^-$ as we will see next. 

If $C$ is not hyperelliptic of genus $g\ge 4$, then
the codimension of $V^2$ in $\mP$ is $3$ (see \cite[Thm. 2.2]{CaLaVi}). In fact, we will see that if $g>4$ then $V^2$ is pure-dimensional, reduced and Cohen-Macaulay (see Lemma \ref{lem:big-comp}).

\subsection{Fourier-Mukai transform}\label{sec:FM}
Let $(A,\Theta)$ be a ppav. Consider $A\times A$ with the corresponding
projections $p_1$ and $p_2$ and $m:A\times A\to A$ the group law. We denote by
$\mathcal{M}=m^\ast\cO_A(\Theta)\otimes p_1^\ast\cO_A(-\Theta)\otimes p_2^\ast\cO_A(-\Theta)$ the Mumford
line bundle. Then we define
\begin{equation*}
 \bR \fm_A (\cF):= \bR{p_1}_\ast(p_2^\ast\cF\otimes \mathcal{M}), 
\end{equation*}
the Fourier-Mukai transform $\bR \fm_A:\mathbf{D}^b(A)\to \mathbf{D}^b(A)$ 
where $\mathbf{D}^b(A)$ is the bounded derived category of coherent sheaves on $A$.
Due to the well-known theorem of Mukai \cite[Thm. 2.2]{Muk81}, this is an equivalence of categories.

By base-change one has an inclusion of sets:
\begin{equation*}
 \supp R^i \fm_{A}(\cF)\subset V^i(\cF):=\{\alpha \in \what A \,\vert \, h^i(\cF \otimes \alpha
) > 0\}.
\end{equation*}
The closed sets $V^i(\cF)$ are called the \emph{cohomological support loci} attached to $\cF$.

A sheaf $\cF$ on $X$ satisfies the Weak Index Theorem (WIT($i_0$)) if there exists $0\le i_0 \le
\dim X$ such that 
 $R^i \fm_A(\cF)=0, \forall i\ne i_0$ and $R^{i_0} \fm_{A}(\cF)\ne 0$. 

We will use the usual dualizing functor in the derived category of a smooth scheme $X$ over a field:
\begin{equation*}
 \bR \Delta_X \cF=\bR \mathcal Hom(\cF,\omega_X).
\end{equation*}
Grothendieck-Verdier duality applied to our context (see \cite[(3.8)]{Muk81}) says:
\begin{equation*}
 \bR \Delta_{ P}\circ \bR \fm_{P}\cong [g-1]\circ (-1)^{\ast}\bR \fm_P \circ \bR \Delta_{P}.
\end{equation*}
Also again by Grothendieck-Verdier duality (e.g. \cite[Thm. 3.34]{huy}) we have:
\begin{equation*}
 \bR \Delta_P \circ j_{\ast}\cong [2-g]\circ j_{\ast} \circ \bR \Delta_{\tC}.
\end{equation*}

\subsection{Theta-duality} Let $(A,\Theta)$ be a principally polarized abelian variety.
Given a morphism $f:T \to A$, we define translation $t_f$ along $f$ to be the composite
\begin{equation*}
t_f:A\times T\stackrel{1\times f}\longrightarrow A\times A\stackrel{m}{\longrightarrow} A,
\end{equation*}
where $m$ is the group law in $A$.

\begin{defn}
The theta-dual of a closed subscheme $Y\subseteq A$ is the unique closed subscheme 
$T(Y) \subseteq A$ with the universal property that an arbitrary morphism 
$f:T \to A$ factors through $T(Y)$ if and only if $Y \times T \subseteq t_f^{-1}(\Theta)$ 
lies in $A \times T$. 
\end{defn}

Thus, as a set, $T(Y)$ consists of those points $a \in A$ for which the theta-translate 
$t_a^{-1}\Theta$ contains $Y$ as a scheme. As proved in \cite[Prop. 2.5]{GulbL}, the 
theta-dual always exists as a scheme and, by \cite[Prop. 2.6]{GulbL} equals the object of the same
name defined in Pareschi--Popa \cite[Def. 4.2]{PPminimal}. That is, the \emph{theta-dual} can be computed as the following
scheme-theoretic support 
\begin{equation*}
 T(Y)=\supp (-1_A)^{\ast} R^a\fm_A(\bR \Delta_A (\cI_Y(\Theta))),
\end{equation*}
where $a=\dim A$.

We also recall \cite[Cor. 4.3]{PPminimal}: \emph{the sheaf} $(-1_A)^{\ast} R^a\fm (\bR
\Delta_A (\cI_Y(\Theta)))$ 
\emph{is a line bundle on} $T(Y)$ \emph{and}
\begin{equation}\label{eq:Thetadual-feix}
 \cO_{T(Y)}(\Theta)\cong (-1_A)^{\ast} R^a\fm (\bR \Delta_A (\cI_Y(\Theta)))\cong \mathcal Hom(\bR\fm_A(\cI_Y(\Theta
)),\cO_A).
\end{equation}
The second isomorphism is a consequence of Grothendieck-Verdier duality (see \cite[Rem. I.1.7]{Tesi} for the
details).

\subsection{Relation between $\bR \fm_{J\tC}(i_{\ast}(-))$ and $\bR \fm_{P}(j_{\ast}(-)) $}
Given a smooth irreducible projective curve, we denote by $i$ an
Abel-Jacobi immersion 
of the curve
in its Jacobian attached to a fixed point.

In \cite[\S III.2.4]{Tesi}, Lahoz proves the following generalization of \cite[Prop. 3.1]{Nar}:
\begin{equation}\label{eq:restr_a_P}
 \bR \fm_{P}(j_{\ast}(-))\cong \bR \fm_{J\tC}(i_{\ast}(-))\underline {\otimes } \cO_P. 
\end{equation}

A similar result is true with respect to $P'$:
\begin{lem} 
We have the following isomorphisms
as functors of the derived category 
of $P'$.
\begin{equation*}
\bR \fm_{J\tC}(i_{\ast}(-))\underline {\otimes} \cO_{P'} \cong
t_{-L'}^\ast(\bR\fm_{P}(j_{\ast}(-\otimes L'))) 
\end{equation*}
where $L'=\cO_{\tC}(\sigma \tilde c - \tilde c)$.

\end{lem}
\begin{proof}
Observe that $\bR \fm_{J\tC}(i_{\ast}(-))\underline {\otimes} \cO_{P'}\cong 
t_{-L'}^\ast(t_{L'}^\ast(\bR \fm_{J\tC}(i_{\ast}(-)))\underline {\otimes} \cO_{P})$ where $L'\in P'$ is
$L'\cong \cO_{\tC}(\sigma \tilde c -\tilde c)$. Then
$\bR \fm_{J\tC}(i_{\ast}(-))\underline {\otimes} \cO_{P'}\cong t_{-L'}^\ast(\bR
\fm_{J\tC}(i_{\ast}(-)\otimes L'))\underline {\otimes} \cO_{P})\cong 
t_{-L'}^\ast(\bR \fm_{J\tC}(i_{\ast}(-\otimes L')))\underline {\otimes} \cO_{P})\cong
t_{-L'}^\ast(\bR\fm_{P}(j_{\ast}(-\otimes L')))$.
\end{proof}

We consider an invertible sheaf on $\tC$ with norm the canonical sheaf and odd $h^0$, i.e. $\mL \in
\mP$.
The sheaf $j_{\ast}\mL$ on $P$ is supported on the curve $\tC$ and its Fourier-Mukai transform
codifies geometric 
information 
of the ppav $(P,\Xi )$ in connection with the subvariety $V^2$. The first observation is that this
sheaf is not WIT 
(compare with the situation when one takes $\pL\in\pP$ studied in \cite[3.4]{Nar}).

Observe that the cohomological support loci attached to $j_{\ast}\mL$ are:
\begin{equation*}
 V^0(j_{\ast}\mL)=V^1(j_{\ast}\mL)=P,\qquad V^i(j_{\ast}\mL)=\emptyset \text{ for }i\ge 2.
\end{equation*}
Moreover $R^0\fm_P(j_{\ast}\mL)$, $R^1\fm_P(j_{\ast}\mL)$ have (generic) rank $1$ and 
$R^i\fm_P(j_{\ast}\mL)=0$ for $i\ge 2$. 

Observe that $R^1 \fm_P(j_{\ast}(\mL)$ is the highest non-vanishing sheaf, so it has the ``base
change
property'', in other words:
\begin{equation*}
 R^1\fm_P(j_{\ast}\mL) \otimes k(\alpha) \cong H^1(\tC, \mL \otimes \alpha) \cong H^0(\tC, \sigma^{\ast}\mL\otimes
\alpha^{-1})^{\ast}. 
\end{equation*}
 Hence it is an invertible sheaf on the complement of (some translate) of $V^2$ in $P$.

\begin{cor}
We have the following sheaf isomorphisms:
\begin{align*}
R^1\fm_{J\tC}(i_{\ast}\mL)_{\mid P} & \cong R^1\fm_{P}(j_{\ast}\mL)\qquad\text{ and }\qquad
\mathcal{T}or_1(R^1\fm_{J\tC}(i_{\ast}\mL),\cO_P) \cong R^0\fm_{P}(j_{\ast}\mL) \\
R^1\fm_{J\tC}(i_{\ast}\mL)_{\mid P'} &\cong
t_{-L'}^\ast(R^1\fm_{P}(j_{\ast}(\mL\otimes L'))) 
\end{align*}
for any $\mL\in P^-\subset \Pic^{2g-2}\tC$. Moreover 
\begin{equation*}
\mathcal{I}_{V^2/P}=\mathrm{Fitt}_2(R^1\fm_{P}(j_{\ast}M^+)), 
\end{equation*}
where $M^+$ is the element used to translate $\Xi^+$ into $P$. 
\end{cor}
\begin{proof}
 Since an element $L\in \Pic^{2g -2}\tC$ satisfies WIT($1$) with respect to $\bR\fm_{J\tC}$
and an element of $L^+$ satisfies WIT($1$) with respect to $\bR\fm_P$, the equalities follow
directly from the previous Lemma. The last statement follows by
definition.
\end{proof}

\section{Scheme theoretical Theta-dual of the Abel-Prym curve}

The aim of this section is to prove the first part of Theorem \ref{ThmA}, namely the scheme
theoretical equality $T(\tC)=V^2$. We assume $g\geq 4$. 
To start with:

\begin{lem}\label{lem:TupC}
We have the set-theoretic equality
\begin{equation*}T(\tC)=V^2.\end{equation*} 
\end{lem}
\begin{proof}
An element $L\in P^-$ belongs to $T(\tC)$ if and only if $\tC\subset
\Xi^+_{-L}$, which, by the definition of $\tC\subset P'$,
is equivalent to $h^0(\tC, L\otimes \cO_{\tC}(\sigma p-p))>0$ for every $p\in \tC$.
By Mumford's parity trick, this happens if and only if $h^0(\tC,L)\geq
3$, that is $L\in V^2$. Indeed, for a generic $p\in \tC$, $h^0(\tC,L)=h^0(\tC, L\otimes \cO_{\tC}(\sigma p-p))+1$
(see \cite[pg. 188, Step II]{Mtheta}). Since $L\in P^-$, then $h^0(\tC, L\otimes \cO_{\tC}(\sigma p-p))>0$ for every $p\in \tC$ if and only if $L\in V^2$.
\end{proof}

The remainder of this section is devoted to proving the following result.
\begin{thm}\label{thm:TupC}\index{theta-dual!of an Abel-Prym curve}
We have the scheme-theoretic equality
\begin{equation*}T(\tC)=V^2.\end{equation*} 
\end{thm}

By Lemma \ref{lem:TupC} we already know that the set theoretical equality holds in $P^-$. 
Now, in order to avoid all the translations that could complicate the Fourier-Mukai argument,
we will work up to translation on $P$ using the fixed isomorphisms.

We know by \eqref{eq:Thetadual-feix}:
\begin{equation*}\cO_{T(\tC)}(\Xi)=\mathcal Hom(\bR \fm_P(\cI_{\tC}(\Xi)),\cO_P),\end{equation*} 
so we want to compute this last sheaf, or construct a comprehensible short exact sequence where it
appears.

Consider $\tilde c$ the point in $\tC$ that we have used to define the Abel-Prym map 
$j$ and
the following short exact sequence in $P$
\begin{equation}\label{eq:ideals}
0\to \cI_{\tC/P}({\Xi})\to \cI_{0/P}(\Xi)\to j_\ast\cI_{{\sigma (\tilde c)}/\tC}({\Xi})\to 0, 
\end{equation}
where we have used that $j(\sigma (\tilde c))=0\in P$. We will work out the Fourier-Mukai transform
of 
this exact sequence 
and its dual. The transforms of the leftmost sheaf will provide the sheaves that we want to
understand. 
The middle sheaf is easy to work with. And the rightmost sheaf can be work out in the Abel-Prym
curve. We get:

\begin{lem} The sheaves $\cI_{0/P}(\Xi)$ and $j_\ast\cI_{\sigma (\tilde c)/\tC}(\Xi)$ satisfy WIT(1).
More precisely
\begin{equation*}
 \bR \fm_P (\cI_{0/P}(\Xi))=\cO_{\Xi}[-1], \qquad 
 \bR \fm_P j_\ast\cI_{\sigma (\tilde c)/\tC}(\Xi)=R^1\fm_P(j_\ast(\sigma^\ast(M^+)(-\tilde c)))[-1];
\end{equation*}
where $M^+\in P^+$ satisfies $\Xi = t_{M^+}^{\ast}(\Xi^+)$.
\end{lem}

\begin{proof}
The statement on $\cI_{0/P}(\Xi )$ comes easily since 
$$\bR \fm_P (\cO_P(\Xi))=\cO_P(-\Xi)[0], \quad
\bR \fm_P (k (0))=\cO_P[0]$$ 
and $R^0 \fm_P(\cI_{0/P}(\Xi ))$ is a torsion sheaf.

On the other hand, by \eqref{eq:JacobiInvPryms}, we have
\begin{equation*}
 j_\ast \cI_{\sigma (\tilde c)/\tC}(\Xi)= j_\ast \cI_{\sigma (\tilde c)/\tC}(j^\ast(\Xi))=j_\ast(\sigma
^\ast(M^+)(-\tilde c)).
\end{equation*}
Therefore, by applying $\bR \fm_P \circ j_\ast$ to the short exact sequence:
\begin{equation*}
 0\to \sigma^\ast (M^+)(-\tilde c) \to \sigma^\ast (M^+) \to k(\tilde c) \to 0,
\end{equation*}
and using that $j_\ast\sigma^\ast M^+$ satisfies WIT($1$) (see \cite[3.4]{Nar}) we are done.
\end{proof}

We apply the Fourier-Mukai transform to \eqref{eq:ideals} and we implement the information obtained
in the Lemma above. We get that 
$R^0 \fm_P(\cI_{\tC /P}(\Xi )) =0$ and the following exact sequence:
\begin{equation*}
 0\to R^1 \fm_P(\cI_{\tC /P}(\Xi ))\to\cO_{\Xi} \to R^1 \fm_P(j_\ast(\sigma^\ast(M^+)(-\tilde c))) \to
R^2 \fm_P (\cI_{\tC /P}(\Xi ))\to 0.
\end{equation*}

By \cite[Thm. 2.2 and Thm. 3.1]{CaLaVi}, the first sheaf of this sequence is supported in codimension at
least $3$, hence it vanishes. 
We get
\begin{cor}
 The sheaf $\cI_{\tC /P}(\Xi )$ satisfies WIT($2$) and there is a short exact sequence
\begin{equation}\label{eq:FMdelsideals}
0\to \cO_{\Xi} \to R^1 \fm_P(j_\ast(\sigma^\ast(M^+)(-\tilde c))) \to R^2 \fm_P (\cI_{\tC /P}(\Xi ))\to
0.
\end{equation}
\end{cor}

An interesting consequence of this corollary is

\begin{cor} The following isomorphism holds
\begin{equation}\label{eq:Ext2R2}
 \mathcal Ext^2(R^2 \fm_P(\cI_{\tC/P}({\Xi})),\cO_P)=\cO_{T(\tC)}(\Xi)
\end{equation}
\end{cor}
\begin {proof}
We apply \eqref{eq:Thetadual-feix} and the Corollary above: 
\begin{align*}
\cO_{T(\tC)}(\Xi)&\cong \mathcal Hom(\bR \fm_P(\cI_{\tC/P}({\Xi})),\cO_P)\\
&\cong \mathcal Hom(R^2 \fm_P(\cI_{\tC/P}({\Xi}))[-2],\cO_P) &
\\
&\cong \mathcal Ext^2(R^2 \fm_P(\cI_{\tC/P}({\Xi})),\cO_P). \qedhere
\end{align*}
\end{proof}

In particular, by \cite[Thm. 2.2]{CaLaVi}, we have
\begin{equation}\label{eq:codimExt2}
\codim_P \mathcal Ext^2(R^2 \fm_P(\cI_{\tC/P}({\Xi})), \cO_{P})\geq 3.
\end{equation}

Now, we apply the functor $\bR Hom(\,\cdot\,,\cO_P)$ to \eqref{eq:FMdelsideals}. We get
\begin{equation}\label{eq:HomHom}
 \mathcal Hom(R^2\fm_P(\cI_{\tC/P}({\Xi})),\cO_P)\cong \mathcal Hom(R^1 \fm_P
(j_\ast(\sigma^\ast(M^+)(-\tilde c)),\cO_P)
\end{equation}
and the exact sequence 
\begin{align}
 0&\to \mathcal Ext^1(R^2 \fm_P(\cI_{\tC/P}({\Xi})),\cO_P)\to 
\mathcal Ext^1(R^1 \fm_P (j_\ast(\sigma^\ast(M^+)(-\tilde c))),\cO_P)\to \cO_{\Xi}(\Xi)\to\notag \\
&\to \mathcal Ext^2(R^2 \fm_P(\cI_{\tC/P}({\Xi})),\cO_P) \to 
\mathcal Ext^2(R^1\fm_P (j_\ast(\sigma^\ast(M^+)(-\tilde c),\cO_{P})\to 0. \label{eq:longTC}
\end{align}

The next step is to prove the vanishing of two of the terms of this exact sequence. This is the
content of 
the next two lemmas.

\begin{lem}
 We have:
\begin{equation*}
 \mathcal Ext^1(R^2 \fm_P(\cI_{\tC/P}({\Xi})),\cO_P)= 0.
\end{equation*}
\end{lem}

\begin{proof}
Observe that 
\begin{equation*}
\mathcal Ext^1(R^2 \fm_P(\cI_{\tC/P}({\Xi})), \cO_{P})=\mathcal Ext^{-1}
(\bR \fm_P(\cI_{\tC/P}({\Xi})), \cO_{P})=R^{g-2}\fm_P\bR \Delta_P (\cI_{\tC/P}({\Xi})). 
\end{equation*}
By base change, as a sets, $\supp R^{g-2} \fm_P\bR \Delta_P (\cI_{\tC/P}({\Xi}))\subseteq
 V^{g-2}(\bR \Delta (\cI_{\tC/P}({\Xi})))=-V^{1}(\cI_{\tC/P}({\Xi}))$ 
(the last equality, follows from Grothendieck-Serre duality), so we have 
\begin{equation}\label{eq:codimExt1}
\codim_P \mathcal Ext^1(R^2 \fm_P(\cI_{\tC/P}({\Xi})), \cO_{P})\geq 3,
\end{equation}
by \cite[Thm. 3.1 and Thm. 2.2]{CaLaVi}.

Since $P$ is smooth, the functor
$\bR \mathcal Hom(\> \cdot\>,\cO_{P})$ is an involution on $\dbc{P}$.
Thus there is a fourth quadrant spectral sequence
\begin{equation*}E^{i,j}_2 := 
\mathcal Ext^i \Bigl((\mathcal Ext^{-j} (R^2 \fm_P(\cI_{\tC/P}({\Xi})), \cO_{P}),\cO_{P}\Bigr)
\Rightarrow H^{i +j} = \begin{cases}R^2 \fm_P(\cI_{\tC/P}({\Xi}))
&\text{if }i+ j=0\\ 0&\text{otherwise}.\end{cases}\end{equation*} 
We have the following
\begin{align*}
\codim_P \supp \mathcal Ext^1(R^2 \fm_P(\cI_{\tC/P}({\Xi})), \cO_{P})&\geq 3,&\text{by }
\eqref{eq:codimExt1}\\
\codim_P \supp \mathcal Ext^2(R^2 \fm_P(\cI_{\tC/P}({\Xi})), \cO_{P})&\geq 3,&\text{by }
\eqref{eq:codimExt2}\\
\mathcal Ext^i(R^2 \fm_P(\cI_{\tC/P}({\Xi})), \cO_{P})&=0\qquad \text{for all }i>2.
\end{align*}
Recall that $\mathcal Ext^l(\cF,\cO_P)=0$ for all $l<\codim_P \supp \cF$. 
Then the previous spectral sequence yields the following exact sequence
\begin{equation}\label{eq:raro?}
 0\to R^2 \fm_P(\cI_{\tC/P}({\Xi}))\to \mathcal Hom(\mathcal Hom(R^2 \fm_P(\cI_{\tC/P}({\Xi})),
\cO_{P}),
\cO_P)\to \mathcal Ext^3(\cO_{T(\tC)}(\Xi),\cO_P)\to 0.
\end{equation}

It is not hard to see that $R^1 \fm_P(j_\ast(\sigma^\ast(M^+)(-\tilde c)))$ has (generic) rank $1$. 
By \eqref{eq:HomHom}, 
$$\mathcal Hom(\mathcal Hom(R^2 \fm_P(\cI_{\tC/P}({\Xi})), \cO_{P}),\cO_P)$$ is a reflexive sheaf of
rank $1$, 
hence a line bundle. Thus, dualizing \eqref{eq:raro?} we get
\begin{equation*}
 \mathcal Ext^1(R^2 \fm_P(\cI_{\tC/P}({\Xi})), \cO_{P})= 0. \qedhere
\end{equation*}
\end{proof}

\begin{lem}
 We have the vanishing
\begin{equation*}
\mathcal Ext^2(R^1\fm_P (j_\ast(\sigma^\ast(M^+)(-\tilde c))),\cO_{P})
= 0.
\end{equation*}
and the isomorphism
\begin{equation*}
\mathcal Ext^1(R^1 \fm_P (j_\ast(\sigma^\ast(M^+)(-\tilde c))),\cO_P)\cong (-1)^\ast R^1\fm_P(j_{\ast
}M^+(\tilde c)).
\end{equation*}
In particular, by the previous Lemma, \eqref{eq:Ext2R2} and \eqref{eq:longTC}, 
$(-1)^\ast R^1\fm_P(j_{\ast}M^+(\tilde c)) \cong \cI_{T(\tC)/\Xi}(\Xi)$. 
\end{lem}
\begin{proof}
Since $j_\ast\sigma^\ast(M^+)$ satisfies WIT($1$) and $k(j(\tilde c))$ satisfies WIT($0$), we get
 from the exact sequence:
\begin{equation*}
 0\to j_\ast(\sigma^\ast(M^+)(-\tilde c)) \to j_\ast\sigma^\ast M^+ \to k(j(\tilde c)) \to 0
\end{equation*}
that $j_\ast(\sigma^\ast(M^+)(-\tilde c))$ satisfies WIT($1$). Therefore 
\begin{equation*}
\mathcal Ext^i(R^1 \fm_P (j_\ast(\sigma^\ast(M^+)(-\tilde c))),\cO_P)\cong 
\mathcal Ext^i(\bR \fm_P (j_\ast(\sigma^\ast(M^+)(-\tilde c)))[1],\cO_P).
\end{equation*}

By Grothendieck-Verdier duality we have
\begin{align*}
\bR \Delta_P ( \bR \fm_P (j_\ast(\sigma^\ast(M^+)(-\tilde c)))[1])
&\cong \bR \Delta_P ( \bR \fm_P (j_\ast(\sigma^\ast(M^+)(-\tilde c))))[-1] \\
&\cong (-1)^\ast\bR \fm_P \bR \Delta_P (j_\ast(\sigma^\ast(M^+)(-\tilde c)))[g-2] \\
&\cong (-1)^\ast\bR \fm_P j_\ast\bR \Delta_{\tC } (\sigma^\ast(M^+)(-\tilde c)) \\
&\cong (-1)^\ast\bR \fm_P j_\ast \sigma^\ast(M^+(\tilde c)).
\end{align*}
This implies the statements of the Lemma.
\end{proof}

\begin{proof}[{End of the proof of Theorem \ref{thm:TupC}}]
Now, we focus on the computation of $R^1 \fm_P (j_\ast M^+(\tilde
c))$. Consider $E$ an effective divisor on $\tC$ of sufficiently high degree $m\gg 0$ whose support 
does not contain $\tilde c$ and the
following diagram
\begin{equation*}
\xymatrix@C=1.pc@R=1.8pc{
 & 0\ar[d] & 0\ar[d]\\
0 \ar[r] & j_\ast(M^+(-E)) \ar[r]\ar[d] & j_\ast M^+ 
\ar[r]\ar[d]& j_\ast\cO_E \ar[r]\ar@{=}[d] & 0\\
0 \ar[r] & j_\ast(M^+(-E+\tilde c)) \ar[r]\ar[d] & j_\ast M^+(\tilde c) \ar[r]\ar[d]&
j_\ast\cO_E
\ar[r] & 0\\
 & k(j(\tilde c))\ar[d]\ar@{=}[r] & k(j(\tilde c)) \ar[d] \\
 & 0&0
}\end{equation*}
Observe that $j_\ast(M^+-E)$ and $j_\ast(M^+(\tilde c)-E)$ satisfy the Index Theorem with index $1$.

Hence, if we apply the Fourier-Mukai transform to the previous diagram we obtain
the following commutative diagram
\begin{equation}\label{eq:RSdiag}
\xymatrix@C=1.pc@R=1.8pc{
&&&&0\ar[d]\\
&&&0\ar[d]&R^0 \fm_P j_\ast(M^+(\tilde c)) \ar[d]\\
 &&& \alpha\ar[d]\ar@{=}[r] &\alpha \ar[d]\\
&0\ar[r]& \cE_0 \ar[r]^M\ar@{=}[d] & \cE_1 \ar[r]\ar[d]& R^1 \fm_P
j_\ast M^+\ar[r]\ar[d]&0\\
0\ar[r]&R^0 \fm_P j_\ast(M^+(\tilde c)) \ar[r] & \cE_0 \ar[r]^N & \cE_2\ar[d] \ar[r]& R^1 \fm_P
j_\ast(M^+(\tilde c))\ar[r]\ar[d] & 0\\
 &&& 0&0
}\end{equation}
where the sheaves $\cE_i$ are locally free of rank $\rk \cE_0=\rk\cE_1=m$ and $\rk \cE_2=m-1$ and
$\alpha\in \Pic^0P$.

Recall that $\cI_{T(\tC)/\Xi}(\Xi)= R^1 \fm_P (j_\ast M^+(\tilde c))$. Observe that the Fitting ideals can be
computed locally since they commute
with arbitrary base change. Choose a covering of $P$ by open affine subsets such that, in any of
these open subsets, the first vertical short exact sequence of diagram \eqref{eq:RSdiag} splits.
Consider an ordered basis of $\cE_1$, such that the last vector generates $\alpha$.
Then, choose a basis of
$\cE_0$ such that in these open sets
the matrices representing $M$ and $N$ are (i.e. choose the counterimage of the generator of
$\alpha$ as the last vector in the basis of $\cE_0$)
\begin{equation*}
M=\left(\begin{array}{c}N\\\hline 0\cdots 0\, f
 \end{array}\right).
\end{equation*}
Hence, $M$ sends the last vector $v=(0,\ldots,0,1)$ in the basis of $\cE_0$ to an element of $\alpha$.
So, $N$ maps $v$ to $0$ and on this basis
\begin{equation*}
N=\left(\begin{array}{ccc|c}\ast&\cdots & \ast&0 \\\vdots&A&\vdots&\vdots\\
\ast&\cdots &\ast& 0
 \end{array}\right). 
\end{equation*}
Then $\cI_{T(\tC)/P}$ is defined locally by the locus where $N$ drops rank, i.e. the $(m-2)\times
(m-2)$ minors of $N$, that coincide with the $(m-2)\times
(m-2)$ minors of $A$ and that, multiplied by $f$, are the same of those of $M$.
So, locally, 
\begin{align*}
\cI_{T(\tC)/P}\cdot (f)=\mathrm{Fitt}_2 \left(R^1 \fm_P (j_\ast M^+\right))=\cI_{V^2/P}.
\end{align*}
Moreover, $\cI_{\Xi/P}=\cI_{V^1/P}=\mathrm{Fitt}_1 (R^1 \fm_P
j_\ast M^+)$ which is locally defined by $(\det M)$. On the other hand, $\cI_{\Xi/P}=\cI_{\supp R^1 \fm_P
j_\ast(M^+(\tilde c))}$ which is locally defined by $(\det A)$. Hence $f=1$ and $V^2=T(\tC)$ scheme-theoretically.
\end{proof}

\section{Odd Prym Picard sheaves and Theta-dual of $V^2$}

In this section we complete the proof of Theorem \ref{ThmA}.  Theorem \ref{ThmB} will be an
easy consequence.

\subsection{Resolutions} From now on, to simplify notation, we put $\Ro :=R^0\fm_P(j_{\ast}\mL)$
and $\Ru :=R^1\fm_P(j_{\ast}\mL)$.
Let $E$ be a reduced divisor on $\tC$ of degree $m\gg 0$.
By applying the functors $j_{\ast}$ and $\bR \fm_P$ to the exact sequence 

\begin{equation*}
0\to \mL \to \mL (E)
\to \mL (E)_{\vert E} \to 0
\end{equation*}
we get 
\begin{equation}\label{eq:resolution_v1}
\begin{aligned}
0\to \Ro \to & R^0\fm_P(j_{\ast}(\mL (E)))
\to R^0\fm_P(j_{\ast}(\mL (E)_{\vert E})) \to
\\
\to \Ru \to & R^1\fm_P(j_{\ast}(\mL (E)))
\to \dots
\end{aligned}
\end{equation}
One easily checks that $R^i\fm_P(j_{\ast}(\mL (E)))=R^i\fm_P(j_{\ast}(\mL(E)_{\vert E}))
=0$ for $i\ge 1$ and that
\begin{equation*}
\begin{aligned}
 &\cE^0:=R^0\fm_P (j_{\ast}(\mL (E))) \\
 &\cG^0:=R^0\fm_P (j_{\ast}(\mL (E)_{\vert E}))
\end{aligned}
\end{equation*}
are locally free sheaves of rank $m$. In other words: $j_{\ast}(\mL (E))$ and 
$j_{\ast}(\mL (E)_{\vert E})$ satisfy WIT(0) and $\cE^0, \cG^0$ are their Fourier-Mukai
transforms.

The exact sequence \eqref{eq:resolution_v1} becomes:

\begin{equation} \label{eq:resolution}
0\to \Ro \to \cE^0
\to \cG^0 
\to \Ru \to 0.
\end{equation}

For future use we remark that $\cG^0$ is particularly simple: since it is the transformation of a
sheaf supported on $E$ (since $E$ is reduced, its support is a collection of distinct points $j(x_k), k=1,\dots,m$ in $P$) then:
\begin{equation}\label{eq:G0}
 \cG^0 = \oplus_{k=1}^{m} \cL_{j(x_k)},
\end{equation}
 where $\cL_{j(x_k)}:=\cO_{P}(\Xi_{j(x_k)}-\Xi) \in \Pic^0 P$.

\begin{lem}\label{lem:R0inv} 
The sheaf $\Ro$ is locally free of rank $1$.
\end{lem}
\begin{proof}
 The resolution \eqref{eq:resolution} shows that $\Ro$ is a $2$-syzygy of rank $1$. If $\Ro$ were not
invertible
this would contradict the Evans-Griffith Theorem on syzygies \cite[Cor. 1.7]{syz}. 
\end{proof}

\subsection{Duality and the sheaf $\Ru$}
As explained in \S\ref{sec:FM}, we have the isomorphisms of functors:
\begin{equation*}
 \bR \Delta_{ P}\circ \bR \fm_{P}\cong [g-1]\circ (-1)^{\ast}\bR \fm_P \circ \bR \Delta_{P}.
\end{equation*}
and:
\begin{equation*}
 \bR \Delta_P \circ j_{\ast}\cong [2-g]\circ j_{\ast} \circ \bR \Delta_{\tC}.
\end{equation*}
All together applied to the element $\mL$ gives:
\begin{align}
 (-1)^\ast\bR \mathcal Hom (\bR \fm_P(j_{\ast}\sigma^\ast\mL),\cO_P)&
\cong \bR \fm_P ( j_{\ast}( \bR \mathcal Hom_{\tC}(\sigma^\ast\mL,\omega_{\tC})))[1] \notag\\
&\cong\bR \fm_P( j_{\ast}\mL)[1] \label{eq:spseq2}
\end{align}
(observe that $\mL\otimes \sigma^{\ast}\mL\cong \pi^{\ast}Nm_\pi(\mL)\cong \pi^{\ast
}\omega_C\cong \omega_{\tC}$).

\begin{lem}\label{lem:Ext1_Ext2_f} Given any $\mL\in P^-$, consider
\begin{align*}
\Ro =R^0\fm_P(j_{\ast}\mL), \quad &\Ru =R^1\fm_P(j_{\ast}\mL)\text{ and }\\
\Ro_\sigma =(-1)^\ast R^0\fm_P(j_{\ast}\sigma^\ast\mL), \quad &\Ru_\sigma =(-1)^\ast R^1\fm
_P(j_{\ast}\sigma^\ast\mL). 
\end{align*}
Then,
\begin{align}
& \mathcal Hom(\Ru,\cO_P)\cong \Ro_\sigma\label{eq:HomR1}\\
& \mathcal Ext^1(\Ru ,\cO_P)=0 \notag\\
& \mathcal Ext^q(\Ru ,\cO_P)=0 \text{ for all }q>3,\text{ and}\notag\\
& 0\to\Ru\to \mathcal Hom (\mathcal Hom (\Ru,\cO_P),\cO_P)\to \mathcal Ext^2(\Ru_\sigma,\cO_P)\to 0 \notag
\end{align}
In particular, $\Ru$ is torsion-free.
\end{lem}

\begin{proof}
The isomorphism \eqref{eq:spseq2} in the derived category of sheaves, induces the following spectral sequence:
\begin{equation*}E^{i,j}_2 := 
\mathcal Ext^i \bigl(((-1)^\ast R^{-j}\fm_P(j_\ast\sigma^\ast\mL),\cO_P\bigr)
\Rightarrow R^{i+j+1} \fm_P( j_{\ast}\mL)\end{equation*} 
This is a fourth quadrant spectral sequence with only $2$ rows and the only non-zero element in the row $j=0$ is
$E^{0,0}_2$ since $\Ro_\sigma$ is locally free. It is easy to deduce from that
\begin{align}
 &\Ro \cong \mathcal Hom(\Ru_\sigma,\cO_P) \notag \\
0\to \mathcal Ext^1(\Ru_\sigma,\cO_P)\to\; &\Ru\to \mathcal Hom (\Ro_\sigma,\cO_P)\to \mathcal
Ext^2(\Ru_\sigma,\cO_P)\to 0 \label{eq:sss}\\
& \mathcal Ext^q(\Ru_\sigma,\cO_P)=0 \text{ for all }q\geq 3. \notag
\end{align}
Since $(-1)^\ast$ and $\sigma^\ast$ are involutions in $\Pic(P)$, we have also the same statements substituting
$\mathcal R^i$ by $\mathcal R^i_\sigma$ and vice versa. Observe that the middle morphism $\Ru\to \mathcal Hom
(\Ro_\sigma,\cO_P)\cong \mathcal Hom (\mathcal Hom (\Ru,\cO_P),\cO_P)$ is the natural map from any sheaf to its double
dual. Then we use the following result on coherent sheaves (see for instance \cite[Prop. 1.1.10]{HL}):
\emph {for a sheaf} $E$
\emph{ on a smooth variety} $X,$ \emph{ the following are equivalent:} 
\begin{enumerate}
 \item $\codim(\supp (\mathcal Ext^q(E,\omega_X)))\ge q+1,\quad \forall q>0$.
 \item $E\hookrightarrow \mathcal Hom (\mathcal Hom (E,\omega_X),\omega_X)$.
\end{enumerate}
In our case $X=P,\omega_X=\cO_P$ and $E=\Ru_\sigma$. We have seen that $\mathcal Ext^q(\Ru_\sigma,\cO_P)=0,
\forall q\ge 3$.
Moreover $\Ru_\sigma$ is locally free out of $V^2$, hence $\supp(\mathcal Ext^q(\Ru_\sigma,\cO_P))\subset V^2$. 
Since $V^2$ has codimension $3$ 
the condition (i) is fulfilled. The condition (ii) implemented in \eqref{eq:sss} completes the proof of the
Lemma.
\end{proof}

\subsection{The sheaf $\Ro$} 
In view of Lemma \ref{lem:R0inv}, $\Ro$ corresponds to a Cartier divisor. We want to determine this divisor modulo algebraic equivalence.
In order to do that, we will consider chern classes as a cycles modulo algebraic equivalence.
We identify $\Pic^0 P$ with $P$ as ppav via the isomorphism induced by $\Xi$.

Going back to \eqref{eq:resolution} and \eqref{eq:G0}, since $c_1(\cG^0)=\Sigma c_1(\cL_{j(x_k)})=0$, we get:
\begin{equation*}
 c_1(\Ro)-c_1(\Ru)=c_1(\cE^0).
\end{equation*}
The isomorphism \eqref{eq:HomR1} allows us to compare $c_1(\Ru)$ and $c_1(\Ro)$. 

\begin{lem} One has: $-c_1(\Ru)=c_1(\Ro)$.
\end{lem}
\begin{proof}
The sheaf $\Ru$ is invertible on the complement of a closed subset of codimension $3$, therefore by \eqref{eq:HomR1}
\begin{equation*}
-c_1(\Ru)=c_1(\Ro_\sigma)
=c_1((-1)^{\ast} R^0\fm_P(j_{\ast} \sigma^{\ast}\mL)).
\end{equation*}
Since in codimension $1$ algebraic and homological equivalence
coincide, we have that 
\begin{equation*}
c_1((-1)^{\ast} R^0\fm_P(j_{\ast} \sigma^{\ast}\mL))=c_1(R^0\fm_P(j_{\ast} \sigma^{\ast}\mL)). 
\end{equation*}
Now we want to see that $c_1(R^0\fm_P(j_{\ast} \sigma^{\ast}\mL))=c_1(\Ro)$.

More precisely, we will see that the algebraic equivalence class of
$c_1(R^0\fm_P(j_{\ast} \mL))$ does not depend on $\mL \in \mP$. 
Let $\mL_1\in \mP$ be another element. Put $\mL_1 \otimes
(\mL)^{-1}=j^{\ast}\alpha$, for
some $\alpha \in \Pic^0(P)$.
Then 
\begin{equation*}
\bR \fm_P(j_{\ast}\mL_1)\cong \bR \fm_P(j_{\ast}\mL\otimes \alpha)\cong 
t_{\alpha }^{\ast}\bR \fm_P(j_{\ast}\mL_1).
\end{equation*}
The last isomorphism is a consequence of the commutation between translation isomorphisms and
tensoring with elements of $\Pic^0 (P)$ 
proved by Mukai \cite[(3.1)]{Muk81}. Now take $0$-cohomology and $c_1$.
\end{proof}

Hence
\begin{equation}\label{eq:c1Ro}
 2 c_1(\Ro)=c_1(\cE^0).
\end{equation}

We use now \cite[Cor. 1.18]{Muk85}: for a WIT($i_0$)-sheaf $\cF$ :
\begin{equation*}
 ch_i(R^{i_0}\fm (\cF))=(-1)^{i+i_0}PD(ch_{a-i}(\cF)),
\end{equation*}
where $a$ is the dimension of the abelian variety and $PD$ stands for the Poincar\'e duality
isomorphism.

Since $\cE^0=R^0\fm_P(j_{\ast}(\mL(E)))$ and $j_{\ast}(\mL(E))$ is WIT(0):
\begin{equation*}
 c_1(\cE^0)=ch_1(R^{0}\fm (j_{\ast}(\mL(E))))=-PD(ch_{g-2}(j_{\ast}(\mL(E)))).
\end{equation*}

The Grothendieck-Riemann-Roch formula says:
\begin{equation*}
\begin{aligned}
 ch(j_{\ast}(\mL(E))\cdot todd (T_P))&=j_{\ast}(ch_{\tC}(\mL(E))\cdot todd (T_{\tC})) \\
&=j_{\ast}((1+\mL (E))\cdot (1-\tfrac{1}{2}K_{\tC})) 
\\
&=[\tC]+\mathrm{deg}(\mL (E)-\tfrac{1}{2}K_{\tC })
\\
&=2\frac{[\Xi ]^{g-2}}{(g-2)!}
+m-g+1.
\end{aligned}
\end{equation*}
Hence:
\begin{equation} \label{eq:c1E0}
 c_1(\cE^0)=-2PD\left(\frac{[\Xi ]^{g-2}}{(g-2)!}\right)=-2[\Xi ].
\end{equation}

\begin{lem}\label{lem:Ro=Theta}
 For some $\lambda \in P$ there is an isomorphism:
\begin{equation*}
 \Ro \cong \cO_P(-\Xi_{\lambda }).
\end{equation*}
\end{lem}
\begin{proof}
Implementing \eqref{eq:c1E0} in \eqref{eq:c1Ro} we reach to
\begin{equation*}
 -2 c_1(\Ro)=2[\Xi ]
\end{equation*}
which easily implies the statement.
\end{proof}

\subsection{Application of Mukai's Theorem}
The main Theorem in \cite[Thm. 2.2]{Muk81} says that:
\begin{equation*}
 \bR \fm_P \circ \bR \fm_P \cong (-1)^{\ast} [-(g-1)].
\end{equation*}
This gives the spectral sequence:
\begin{equation*}
E^{ij}_2:=R^i\fm_P(R^j \fm_P(j_{\ast}\mL)) \Longrightarrow H^{i+j}:=\left\{\begin{array}{ll} (-1)^{\ast}j_{\ast}\mL
&\text{if }i+j=g-1\\ 0&\text{otherwise}.
\end{array}\right.
\end{equation*}

Observe that $E^{ij}_2=0$ if $i\ge g$ or $j\ge 2$. Hence everything vanishes except two rows. 
Putting $R^{ij}:=E^{ij}_2$ we obtain the following information:

a)  Observe that, by Lemma \ref{lem:Ro=Theta},
$R^{g-1,0}=R^{g-1}\fm_P(\cO_P(-\Xi_{\lambda}))\cong (-1)^{\ast}\cO_P (\Xi_{\lambda})$ is the only non-zero element for $j=0$. Then, all the differentials from $R^{i,1}$ vanish unless $i=g-3$. If $i+j\neq g-1$, then $E_\infty^{i,j}=0$. Hence, the only non-zero elements for $j=1$ are $R^{g-3,1}$ and $R^{g-2,1}$.

b) There is an exact sequence:
\begin{equation}\label{eq:espectral}
 0\to R^{g-3,1} \to (-1)^{\ast}\cO_P(\Xi_{\lambda }) \to (-1)^{\ast}j_{\ast}\mL \to
R^{g-2,1}\to 0.
\end{equation}

In particular $(-1)^{\ast}R^{g-3,1}$ is an ideal sheaf $\cI_Z$ twisted with $\Xi_{\lambda}$. 
Therefore, \eqref{eq:espectral} provides a short exact sequence
on $\tC$:
\begin{equation*}
 0\to (-1)^{\ast}\cO_Z(\Xi_{\lambda }) \to (-1)^{\ast}j_{\ast}\mL \to R^{g-2,1}\to 0.
\end{equation*}
Hence $Z=\tC$ and $(-1)^{\ast}R^{g-3,1}\cong \cI_{\tC} (\Xi_ {\lambda})$. Moreover
$\cO_{\tC}(\Xi_{\lambda })$ has degree $2g-2=\deg \mL$, so $R^{g-2,1}=0$.

 \begin{cor}\label{cor:Ru_WIT}
The sheaf $\Ru$ is WIT($g-3$) and its Fourier-Mukai transform is $\cI_{\tC}(\Xi_{\lambda })$. 
 \end{cor}

\begin{rem}
Since the projectivization of the sheaf $\bR \fm_{J\tC}(i_{\ast}(\mL))$ is $\tC^{(2g-2)}$, by \eqref{eq:restr_a_P}
the projectivization of $\Ru$ is $\mX$. Then one could try to recover directly $\Ru$ from $X^-$ imitating the
argument in \cite[\S4]{Nar}. In this case, $\Ru$ is locally free on $P\setminus V^2\stackrel{i}\hookrightarrow P$, but it is
easy to see that $i_\ast i^\ast\Ru\neq \Ru$, and this fact invalidates this direct strategy.
\end{rem}

\subsection{The scheme-theoretic equality $T(V^2)=\tC$}
The equality $T(\tC)=V^2$ obtained in the last section combined with \eqref{eq:Thetadual-feix} gives:
\begin{cor}
 There is an isomorphism of invertible sheaves on $V^2$:
\begin{equation}\label{eq:TC_feixos}
\mathcal Hom(\bR \fm_P(\cI_{\tC}(\Xi_{\lambda})),\cO_P)\cong \cO_{V^2}(\Xi_{\lambda }).
\end{equation}
\end{cor}

Isomorphism \eqref{eq:TC_feixos} in combination with Corollary \ref{cor:Ru_WIT} allows us to prove the following
result:

\begin{prop}\label{prop:Ru=IV}
 We have:
\begin{equation*}
 (-1)^{\ast}\mathcal Ext^2(\Ru,\cO_P)\cong \cO_{V^2}(\Xi_{\lambda}).
\end{equation*}
In particular, by Corollary \ref{cor:Ru_WIT}, $\Ru\cong \cI_{V^2}(\Xi_{\lambda}).$
\end{prop}

\begin{proof}
By Corollary \ref{cor:Ru_WIT} we can replace in \eqref{eq:TC_feixos} the sheaf $\cI_{\tC}(\Xi_{\lambda})$ by
$R^{g-3}\fm_P(\Ru)=\bR \fm_P(\Ru)[g-3]$. So:
\begin{equation*}
 \cO_{V^2}(\Xi_{\lambda})\cong \mathcal Hom((-1)^{\ast}(\Ru [1-g+g-3],\cO_P)\cong (-1)^{\ast}
\mathcal Ext^2(\Ru, \cO_P). \qedhere
\end{equation*}
\end{proof}

\begin{cor}
 The scheme-theoretic equality $T(V^2)=\tC$ holds.
\end{cor}

\begin{proof}
In  Proposition \ref{prop:Ru=IV} we have obtained an expression of the ideal of $V^2$ (tensored with
the Theta divisor) in terms of
$\Ru $. By replacing in the definition of $T(V^2)$ we get:
 \begin{align*}
 T(V^2)&= \supp(\mathcal Hom(\bR \fm_P(\Ru ),\cO_P) \\
&= \supp(\mathcal Hom(\cI_{\tC}(\Xi_{\lambda})[(g-3)],\cO_P)) &\text{ see Cor. \ref{cor:Ru_WIT}} \\
&= \supp( \mathcal Ext^{g-3}(\cI_{\tC}(\Xi_{\lambda}),\cO_P)))\\
&= \supp(\mathcal Ext^{g-2}(\cO_{\tC}(\Xi_{\lambda}),\cO_P)))\\
&=\tC.
 \end{align*}
The last equality is standard and can be obtained for example by noticing that $\mathcal Ext
^{g-2}(\cO_{\tC}(\Xi_{\lambda}),\cO_P)))$
is the dual sheaf of $\cO_{\tC}(\Xi_{\lambda})$ in $P$ and that duality commutes with $j_{\ast}$
(c.f. \cite[pg. 5]{HL}).
\end{proof}

\begin{rem}
The scheme-theoretic equalities $T(T(\tC))=\tC$ and $T(T(V^2))=V^2$ tell us that $\tC$ and $V^2$ are scheme-theoretically intersections of translates of $\Xi$. 
\end{rem}

\section{Reduceness of $V^2$ and the proof of Theorem \ref{ThmB}}
\subsection{$V^2$ is reduced}
\begin{lem}\label{lem:big-comp}
For any \'etale double cover $\tC\to C$, such that $C$ is a non-hyperelliptic curve of genus $g>4$, $V^2$ is reduced, pure-dimensional 
and Cohen-Macaulay. 
\end{lem}
\begin{proof} 
\begin{description}
\item[Step 1. If $\dim V^2>0$, then $V^2\setminus V^4$ is generically smooth]
Let $L \in V^2\setminus V^4$, so $h^0(\tC,L)=3$. Suppose that $V^2$
is singular at $L$, so 
\begin{equation}\label{eq:hypo}
\dim T_L V^2 >g-4.
\end{equation}

The Zariski tangent space $T_L V^2$ is given as the orthogonal
complement to the image of the map (see \cite[(1.9)]{WENS})
\begin{equation*}v_0:\wedge^2H^0(\tC,L)\rightarrow {H^0(\tC,\omega_{\tC})}^-
\end{equation*}
defined by $v_0(s_i\wedge s_j)=s_i\sigma^\ast s_j-s_j\sigma^\ast s_i.$

The inequality \eqref{eq:hypo} is equivalent to $\dim(\mathrm{im}\, v_0 ) < 3$. On the other hand, all the
forms in $\wedge^2 H^0 (\tC, L)$ are decomposable and $\dim \wedge^2 H^0 (\tC, L)=3$,
so there is a decomposable form $s_i\wedge
s_j$ in $\ker v_0$. This means that $s_i\sigma^\ast s_j-s_j\sigma^\ast s_i=0$, or in
other words that $\tfrac{s_j}{s_i}$ defines a rational function
$h$ on $C$. This induces a map $C\to \mathbb{P}^1$, so there exists a line bundle $M$, with $h^0(M)\geq 2$, 
such that $L\cong\pi^\ast M\otimes \mathscr O_{\tC}(B)$
and $B$ is an effective divisor (the maximal common divisor between $(s_i)_0$ and $(s_j)_0$).

Observe that $Nm(B)\in |K_C\otimes M^{\otimes -2}|$ and for any $L \in V^2\setminus V^4$ such that $V^2$ is singular
at $L$, we obtain a pair $(M,B)$. The family of such pairs $(M,B)$ is a finite cover of the set of pairs $(M, F)$
where:
\begin{itemize}
\item $M$ is an invertible sheaf on $C$ of degree $d\geq 2$ such that $h^0(C,M)\geq 2$,
\item $F$ is an effective divisor on $C$ of degree $2g-2-2d\geq 0$, such that
$F\in |K_C\otimes M^{\otimes -2}|$.
\end{itemize}

By Martens' theorem applied to the non-hyperelliptic curve $C$
(see \cite[pg. 192]{ACGH}),
the dimension of the above family of line bundles $M$
is bounded above by
\begin{equation}\label{eq:dis1}
{\rm dim}(W^1_d)\leq d-3.
\end{equation}

Fixing a line bundle $M$ as above, the dimension of possible $F$ satisfying the second condition is bounded by
Clifford's theorem (e.g. \cite[pg. 107]{ACGH}),
\begin{equation}\label{eq:dis2}
h^0(K_C\otimes M^{\otimes -2})-1 \leq g-1-d.
\end{equation}
Moreover, equality holds if and only if $M$ is a two torsion point or a theta-characteristic. The first possibility is excluded since $h^0(C,M)\geq 2$. If $M$ is a theta-characteristic, then $F=0$ and there is only a finite number of pairs $(M,F)$. Hence, a finite number of singular points on $V^2\setminus V^4$.

When equality does not hold, i.e. $h^0(K_C\otimes M^{\otimes -2})-1 \leq g-2-d$, together with inequality \eqref{eq:dis1}, we get that the dimension $m$ of our family of pairs
$(M, F)$ is bounded above by $m\leq d-3 +g-2-d=g-5$. So $V^2\setminus V^4$ is singular at most in codimension
$4$ in the Prym variety.

In both cases we have that, if $V^2$ is positive dimensional, then $V^2\setminus V^4$ is generically smooth.
\item[Step 2. $V^2\setminus V^4$ has non-empty intersection with every component of $V^2$] Suppose that a component of
$u^{-1}V^2$ is entirely contained in $u^{-1}V^4$. Let $D$ be a general point of a component of $u^{-1}V^2$ and assume
that $D\in u^{-1}V^4$. Then, for every $p,q\in \tC$, $h^0(D-p-q)>0$ and if they are not base
points of $|D|$, then $h^0(D-p-q)=h^0(D)-2>0$. Take $E\in|D-p-q|$.  By the parity trick of Mumford \cite[pg. 188, Step II]{Mtheta} $h^0(D-p-q+\sigma
p+\sigma q)=h^0(D)-2$ and by the generality of $p$ and $q$, $E+\sigma p+\sigma q$ belongs to the same
component of $u^{-1}V^2$ as $D$, which contradicts by semicontinuity the generality of $D$.

\item[Step 3. $V^2$ is reduced]  
By the previous steps, to show that $V^2$ is reduced is enough to see that it has no
embedded components. If $V^2$ is Cohen-Macaulay this is straightforward. $V^2$ is Cohen-Macaulay and equidimensional
if,  and only if, $\mathcal Ext^k(\cO_{V^2},\cO_P)\neq 0 \Leftrightarrow k= 3$. Since $\Ru\cong
\cI_{V^2}(\Xi_{\lambda})$, by Proposition \ref{prop:Ru=IV}, this follows from Lemma \ref{lem:Ext1_Ext2_f}.\qedhere
\end{description}
\end{proof}

\begin{rem}
When $C$ is a generic curve of genus $4$, $V^2$ consists of two different points. 
However it could happen that $V^2$ is not reduced for some coverings. If that case
 the curve $C$
would have a unique $g^1_3$ complete linear series.
\end{rem}

\begin{rem}
 The authors have been informed by A. H\"oring that he has obtained this result 
in \cite[Thm. 1.1]{Ho2} for $g\ge 6$. In fact, 
he also studies when  $V^2$ is reducible. One of these cases is the intermediate Jacobian of the cubic threefold that we
consider in the last section of this paper. Since $V^2$ has the expected dimension, the fact that $V^2$ is Cohen-Macaulay 
(shown in Step 3) is also proved in \cite[(6.2) 1)]{Debarre2}. We keep our proof here for the convenience of the reader 
since it is a rather direct consequence of our computations. 
\end{rem}

In particular, $V^2$ is pure-dimensional and, as Corollary of \cite[Thm. 9]{dCP}, we have:
\begin{cor} The cohomological class of $V^2$ is
\begin{equation*}
[V^2]=2\frac{\Xi^3}{6}. 
\end{equation*}

\end{cor}

\subsection{Proof of Theorem \ref{ThmB}}
Now one can easily prove Theorem \ref{ThmB}. Let $\mu:\widetilde\mX\to \mX$ be a desingularization of $\mX$. 
The composition of $\mu$ with $\mX \to P$ provides a birational morphism from the smooth variety $\widetilde \mX$ to an abelian variety $P$, which is necessarily the Albanese map of $\widetilde\mX$. 
Thus, the Albanese map of any desingularization of $\mX$ factorizes through $\mX$ and gives a well-defined map $\mX \to P$ (up to translation). 
The locus where the fibers of this map have positive dimension (considered with its reduced structure) is $V^2$. Then, we can recover also the divisor $\Xi$
(up to translation), since there is only one principal polarization such that the cohomology class of $V^2$ is
$2\frac{\Xi^3}6$ (cf. \cite[pg. 288-289]{Ran}). 
Thus, we can consider the theta-dual of $V^2$, $\tC =T(V^2)$ and recover the curve $\tC$
naturally embedded in the Prym variety.
Now, a well-known argument of Welters \cite[(2.2), pg. 96]{Wtwice} 
allows us to recover the involution $\sigma^\ast$ in $\Pic^0\tC$. 
Since $\tC$ is non-hyperelliptic, by the strong Torelli Theorem, the involution $\sigma$ in $\tC$ is recovered.

\section{$V^2$ in the intermediate Jacobians of cubic threefols}
The aim of this section is to illustrate Theorem \ref{ThmA} in a very well-known context: the intermediate Jacobian
$J$ of a smooth cubic threefold $V$, since $J$ is isomorphic (as ppav) to a Prym variety of dimension $5$. We follow
closely the notation and results of \cite{Beauville} and \cite{CG}. 
Let us fix some terminology and recall some basic facts: 
Let $F$ be the Fano surface of the lines contained in $V$. Fixing a point in $F$, the Albanese
map gives an embedding $F\hookrightarrow Alb(F)\cong J$. Let $t\in F$ be  general and denote by 
$C_t$ the (smooth irreducible of genus $11$) curve of the lines
of $F$ intersecting $t$. Observe that this curve comes equipped with two natural structures:
\begin{itemize}
 \item [a)] An involution $\sigma $: if $r\in C_t$, the $2$-plane $t\vee r$ intersects $V$ in $3$ lines, $t,r$ and
            $\sigma r$. The quotient 
            $C_t/\sigma =:D_t$ is a smooth plane quintic. It turns out that $P(C_t,D_t)\cong J$.
 \item [b)] A complete $g^1_5$ attached to the map $C_t \rightarrow t$ mapping $r$ to the 
            intersection point $r\cap t$. Denote by $L\in \Pic^5(C_t)$ the corresponding line bundle. 
            Beauville proves \cite[\S 4.C]{Be} that there are only $2$ 
            such linear series with norm $\cO(1)$ on $C_t$: $L$ and $\sigma L$.
\end{itemize}
Following Beauville, we denote by $D(r)\in C_t^{(5)}$ the divisor $C_t\cap C_r$, i.e. the five lines
that intersect both $t$ and $r$. Observe that if $r$ intersects $t$, then $\cO_{C_t}(D(r))\cong L(\sigma r-r)$.
All these divisors have norm in the linear system $\vert \cO(1) \vert$ of $D_t$. So it is natural to consider 
the special subvarieties attached to $ \vert \cO(1) \vert$ by taking the following fiber product diagram:
\begin{equation*}
\xymatrix{
 S_0\cup S_1 \ar[d] \ar@{^{(}->}[r] &C_t^{(5)} \ar[d] \\
\vert \cO(1) \vert \ar@{^{(}->}[r]&D_t^{(5)}. }
\end{equation*}
The divisors $D(r)$ belong to one of these components, say $S_0$.
Then the map $F\setminus \{t\}\rightarrow S_0$ extends to the blow-up $F_t$ of $F$ at $t$, and
$F_t \rightarrow S_0$ is an isomorphism (cf. \cite[Prop. 3]{Beauville}). The exceptional divisor maps to $\vert \sigma L
\vert$.
The composition $F_t \rightarrow S_0 \longrightarrow P\cong J$ factors through the 
Albanese map $F \hookrightarrow J$.
\begin{lem} With the same notation:
 \begin {itemize}
  \item [a)] $V^3=\{\pi^\ast(\cO(1))\}=\Sing(\Xi)$, $V^4=\emptyset$
  \item [b)] $h^0(C_t,L^{\otimes 2})=h^0(C_t,\sigma L^{\otimes 2})=3$
  \item [c)] For all $r\in F\setminus \{t\}$, $h^0(C_t,\sigma L(D(r)))=h^0(C_t,L(\sigma D(r)))=3.$ 
\end {itemize}
\end{lem}
\begin{proof}
The statement on $V^3$ is proved in \cite[Prop. 2]{Beauville} (in fact it implies that $J$ is not a product of
Jacobians  and therefore $V$ is not rational). If $V^4$ were non empty, then by twisting with 
$r-\sigma r$ we will produce infinite elements in $V^3$. To prove b) observe that $L^{\otimes 2}$ has norm
$\cO(2)\cong \omega _{D_t}$ and it is obtained from
$\pi^\ast(\cO(1)) =L\otimes \sigma L\in P^+$ by twisting with $5$ elements of the form $x-\sigma x$, hence it belongs
to $P^-$. Since $2=h^0(C_t,L)\le  h^0(C_t,L^{\otimes 2})$ and $V^4$ is empty the statement follows. By applying
$\sigma$ the same is true for $\sigma L^{\otimes 2}$. Finally c) is a consequence of b) since 
$\sigma L(D(r))$ specializes to $\sigma L^{\otimes 2}$.  
\end{proof}

From now on we think of $F$ as a surface embedded in the canonical model of $P(C_t,D_t)\cong J$:
\begin{equation*}
 P^+=\{M\in Pic^{10}(C_t) \,\vert \, Nm_{\pi }(M)\cong \cO_{D_t}(1), h^0(M) \text { even}\}
\end{equation*}
 by means of the map $r\mapsto \cO_{C_t}(D(r))\otimes L$.
Put $F_a:=a+F= t^{*}_{-a} F$ and $(\sigma F)_a :=a + \sigma F = t^{*}_{-a}(\sigma F)$.  
By using the previous lemma and comparing the cohomology class of $V^2$ with that of $F$
 the following equality in $P^-$ is straightforward:
\begin{cor} We have
 $V^2=F_{\sigma L\otimes  L^{-1}}\cup (\sigma F)_{ L\otimes \sigma L^{-1}}$.
\end{cor}
The main theorem of \cite{Ho1} combined with the results in \cite[\S 8.1]{PPminimal} imply that 
\begin{equation*}
\begin{array}{l}
 T(F_{\sigma L\otimes L^{-1}})=(\sigma F)_{\sigma (L^{- 2})}=(-1)^\ast(F_{L^{- 2}}) \subset P' \\
 T((\sigma F)_{ L\otimes  \sigma L^{-1}})= F_{L^{- 2}}\subset P'.
\end{array}
\end{equation*}

Hence
\begin{equation*}
 T(V^2)=
F_{L^{- 2}}\cap (-1)^\ast(F_{L^{- 2}}).
\end{equation*}
The last piece that completes the picture is the observation that $\cO_{C_t}(D(r)+D(\sigma r))\cong L^{\otimes 2}$
for any $r$ intersecting $t$, i.e.~for every $r\in C_t$. This says that:
\begin{equation*}
\begin{array}{rcl}
 C_t &\hookrightarrow & T(V^2)\\
r &\mapsto &  L^{-1}(D(r)). 
\end{array}
\end{equation*}
Theorem \ref{ThmA} shows that this inclusion $C_t \subset T(V^2)$ 
is in fact an equality of schemes. 

Summarizing we get
\begin{prop} With the notation above, 
 $T(V^2)= F_{L^{- 2}}\cap (-1)^\ast(F_{L^{- 2}}) = C_t$.
\end{prop}

\end{document}